\documentclass[11pt,reqno]{amsart}
\usepackage[comma,numbers]{natbib}
\usepackage[a4paper,left=1in,right=1in,top=1.25in,bottom=1.25in]{geometry}
\usepackage{color}
\usepackage{graphicx}
\usepackage{amsmath}

\newtheorem{theorem}{Theorem}[section]
\newtheorem{corollary}[theorem]{Corollary}

\theoremstyle{definition}

\theoremstyle{remark}

\def\numberlikeadb{\global\def\theequation{\thesection.\arabic{equation}}}
\numberlikeadb

\newcommand{\eqa}{\begin{eqnarray}}
\newcommand{\ena}{\end{eqnarray}}
\newcommand{\eq}{\begin{equation}}
\newcommand{\en}{\end{equation}}
\newcommand{\eqs}{\begin{eqnarray*}}
\newcommand{\ens}{\end{eqnarray*}}

\def\ignore#1{}

\begin{document}
\title[]{Limiting the spread of disease through altered migration patterns}
%\thanks{This research is supported in part by the Australian Research Council (Centre of Excellence for Mathematical and Statistical Frontiers, CE140100049)}
\maketitle
\noindent R. McVINISH, P.K. POLLETT and A. SHAUSAN \\
School of Mathematics and Physics, University of Queensland \\

% ----------------------------------------------------------------

\noindent ABSTRACT. \ 
We consider a model for an epidemic in a population that occupies geographically distinct locations. The disease is spread within subpopulations by contacts between infective and susceptible individuals, and is spread between subpopulations by the migration of infected individuals. We show how susceptible individuals can act collectively to limit the spread of disease during the initial phase of an epidemic, by specifying the distribution that minimises the growth rate of the epidemic when the infectives are migrating so as to maximise the growth rate. We also give an explicit strategy that minimises the basic reproduction number, which is also shown be optimal in terms of the probability of extinction and total size of the epidemic.  \\

% ----------------------------------------------------------------

\noindent\emph{Key words}:  basic reproduction number; branching process; expected total size; minimax optimisation \\
\noindent\emph{MSC 2010}:  92D30; 90C47; 60J85 

\section{Introduction}
Recently, a number of papers have been devoted to the issue of controlling disease outbreaks. Typical mechanisms for control involve treatments which speed recovery \citep{NMG:11,RLG:09}, culling of infected individuals \citep{NMG:10}, reducing the density of disease vectors \citep{MYOS:14}, vaccination programs \citep{KBMG:12,KLG:11} and quarantine \citep{RLG:09}. When the population has some spatial structure, migration also plays an important role in disease spread and provides a further control mechanism.  

A common approach to incorporating spatial structure in epidemic modelling is to impose a metapopulation structure on the population \citep[see][for example]{DBR:07,GH:97,GS:09,Hess:96}. In a metapopulation, the population is divided into a number of subpopulations occupying geographically distinct locations. The disease is spread within a subpopulation by contacts between infective and susceptible individuals and is spread between subpopulations by the migration of infected individuals.

The effect of migration rates on disease spread in metapopulations has been investigated in a number of papers. Due to the complexity of these models, control strategies are often based on minimising the basic reproduction number $ R_{0}$. Studying a multi-patch frequency dependent SIS model, \citet{ABLN:07} note that the rapid movement of infective individuals can lead to disease extinction in low risk environments. Furthermore, they conjecture that $ R_{0} $ is a decreasing function of the diffusion rate for infective individuals. \citet[Theorem 4.2]{HvdDW:07} note a similar result for their two-patch SEIRP model and a similar phenomena has been observed in  population models with spatially heterogeneous environments \citep{Hastings:83}.  However, \citet[section 4]{GR:12} have shown that for other models the dependence of $ R_{0} $ on migration rates can be more complex. To investigate the effect of the migration rates on other quantities such as the number of infected individuals, numerical methods are generally required \citep[for example][]{SNvGR:12}.

In this paper, we examine how susceptible individuals can act collectively to limit the spread of disease during the initial phase of an epidemic. More specifically, we consider how susceptible individuals can distribute themselves in the metapopulation in a way that minimises the growth of the epidemic when the infectives migrate so as to maximise the growth. By formulating the problem as a minimax optimisation and focusing on the susceptible individuals, we avoid the need to distinguish between infected and susceptible individuals when applying controls to the population. This is advantageous as identification of infected individuals can be problematic due to factors such as delays in the onset of symptoms, asymptomatic carriers and costs associated with testing. Furthermore, acute disease can have a significant effect on the behaviour of animals \citep{Hart:88}. This is particularly true for certain parasitic diseases where the parasite attempts to force the host to act in a manner which assists the propagation of the parasite \citep{Adamo:13}. 

In Section 2 we give our main results. Instead of using an ODE model for the epidemic as was done in the papers cited above, our analysis is based on a branching process model. Branching processes are known to provide a good approximation to the standard SIR and SIS Markov chain models  when the number of infectives is initially small \citep{Clancy:96}. Using this model, we are able to give an explicit strategy that minimises the expected rate of growth under a certain condition on the recovery and infection rates. We also give an explicit strategy that minimises the basic reproduction number which does not require this extra condition. This later strategy is shown to also be optimal in terms of the probability of extinction and total size of the epidemic. In Section 3, the problem of minimising the expected growth rate is investigated numerically. The paper concludes with a discussion of how the results depend on contact rates and how they relate to ODE models.

\section{Minimising disease spread in the initial stages} \label{Sec:Main}

Consider a closed population of size $ N $ divided into $ m $ groups such that at time~$ t $ group~$ i $ contains $ X_{i}(t) $ susceptibles and $ Y_{i}(t) $ infectives. Each individual, conditional on its disease status, moves independently between groups according to an irreducible Markov process on $ \{1,\ldots,m\} $ with transition rate matrix $ R $ if it is susceptible and transition rate matrix $ Q $ if it is infected. The epidemic evolves as a Markov process. Contacts between individuals in the same group are assumed to be density dependent \citep{BBBFHT:02}.  More precisely, a pair of individuals  in group $ i $ makes contact at the points of a Poisson process of rate $ \beta_{i}/ N $ with contacts between distinct pairs of individuals being mutually independent. It is assumed that contact between an infective and a susceptible results in the infection of the susceptible. An infected individual in group $ i $ recovers with immunity at a rate~$ \gamma_{i} $.  Since we are primarily concerned with the initial phase of the epidemic, our conclusions remain valid for epidemics where individuals recover without immunity.

In the absence of infective individuals, the entirely susceptible population evolves following a closed (linear) migration process with per-capita migration rates~$ R $. If the population is in equilibrium, then the probability that an individual is in group~$ i $ is given by $ \pi_{i} $ where $ \pi $ is the unique solution to $ \pi R = 0 $ subject to the constraint $ \pi {\bf 1} = 1 $. 

We consider the spread of the disease from a small number of initial infective individuals. \citet[Theorem 2.1]{Clancy:96} shows that, when $ N $ is large, the epidemic can be approximated by a multi-type branching process. Assuming the susceptible population is in equilibrium, the branching process for the number of infective individuals is given by
\begin{align}
(Y_{1},\ldots,Y_{m}) & \rightarrow (\ldots,Y_{i}+1, \ldots, Y_{j} - 1, \ldots)  & \mbox{at rate } & Q_{ji} Y_{j}, \label{Model:Eq1}\\
(Y_{1},\ldots,Y_{m}) & \rightarrow (\ldots,Y_{i} + 1, \ldots) & \mbox{at rate } & \beta_{i}  \pi_{i} Y_{i}, \label{Model:Eq2}\\
(Y_{1},\ldots,Y_{m}) & \rightarrow (\ldots,Y_{i} - 1, \ldots) & \mbox{at rate } & \gamma_{i}Y_{i}. \label{Model:Eq3}
\end{align}
Note that the branching process depends on $ R $ only through the equilibrium distribution $ \pi $. 

Suppose that the susceptible population aims to minimise some quantity $ f(\pi,Q) $, calculated from the branching process determined by (\ref{Model:Eq1})-(\ref{Model:Eq3}). Let $ \mathcal{S} $ denote the relative interior of the $ (m-1) $-simplex and let $ \mathcal{Q} $ be the set of irreducible migration rate matrices. Without imposing any constraints on the movements of the infectives, the susceptible population can choose $ \pi $ such that, for any $ \epsilon > 0 $, a value no larger than $ \inf_{\pi\in\mathcal{S}} \sup_{Q \in \mathcal{Q}} f(\pi,Q)  + \epsilon $ is attained. On the other hand, the infectives can migrate in such a way that, for any $ \epsilon > 0 $, a value no smaller than $ \sup_{Q \in \mathcal{Q}} \inf_{\pi\in\mathcal{S}} f(\pi,Q) - \epsilon $ is attained. In general,
$$
\sup_{Q\in\mathcal{Q}} \inf_{\pi\in\mathcal{S}} f(\pi,Q) \leq \inf_{\pi\in\mathcal{S}} \sup_{Q\in\mathcal{Q}} f(\pi,Q)
$$
\citep[Lemma in section 1.2.2]{PZ:96}. A pair $ (\pi^{\ast},Q^{\ast}) \in \mathcal{S}\times\mathcal{Q} $ such that 
$$
f(\pi^{\ast},Q) \leq f(\pi^{\ast},Q^{\ast}) \leq f(\pi,Q^{\ast}),
$$
for all $ \pi \in \mathcal{S} $ and all $ Q \in \mathcal{Q}$ is called a saddle point for $ f $. If a saddle point exists, then
$$
\min_{\pi\in\mathcal{S}} \sup_{Q \in \mathcal{Q}} f(\pi,Q) = \max_{Q\in\mathcal{Q}} \inf_{\pi\in\mathcal{S}} f(\pi,Q)
$$
\citep[Theorem in section 1.3.4]{PZ:96}. The susceptibles can attain this value by distributing themselves amongst the groups according to $ \pi^{\ast} $. When a saddle point for $ f $ does not exist, there may still be an $ \epsilon$-saddle point, that is, for every $ \epsilon > 0 $ there exists a pair $ (\pi^{\epsilon}, Q^{\epsilon}) \in \mathcal{S}\times\mathcal{Q} $ such that 
$$
f(\pi^{\epsilon},Q) - \epsilon \leq f(\pi^{\epsilon},Q^{\epsilon}) \leq f(\pi,Q^{\epsilon}) +\epsilon,
$$
for all $ \pi \in \mathcal{S} $ and all $ Q \in \mathcal{Q}$. The existence of an $ \epsilon$-saddle point implies that  
$$
\inf_{\pi\in\mathcal{S}} \sup_{Q \in \mathcal{Q}} f(\pi,Q) = \sup_{Q\in\mathcal{Q}} \inf_{\pi\in\mathcal{S}} f(\pi,Q) = \lim_{\epsilon\rightarrow 0} f(\pi^{\epsilon},Q^{\epsilon})
$$
\citep[Theorem in section 2.2.5]{PZ:96}. In the following, we determine the $(\epsilon)$-saddle points for four quantities derived from the branching process (\ref{Model:Eq1})-(\ref{Model:Eq3}).

As mentioned in the introduction, this formulation avoids the need to distinguish between susceptible and infected individuals in the application of controls. To illustrate this point, suppose that susceptible individuals normally move between groups following a Markov process with migration rate matrix $ R $. The optimal distribution of susceptibles $ \pi^{\ast} $ can be obtained by border controls where a migrating individual from group $ j $ going to group $ i $ is given admittance with probability $ p_{ji} $ and otherwise returned to group $ j $. Detailed balance equations show that the optimal distribution for susceptibles is obtained if the admittance probabilities satisfy
$$
R_{ji} p_{ji} \pi^{\ast}_{j} = R_{ij} p_{ij} \pi^{\ast}_{i},
$$
for all $ i,j $. Although  the border controls will have an effect on the migration rate of infected individuals if they are applied to the population as a whole, the optimal distribution for susceptible individuals ensures that the growth of the epidemic can be no greater than $ \min_{\pi\in\mathcal{S}} \sup_{Q\in\mathcal{Q}} f(\pi,Q) $. 

\subsection{Minimising the expected growth rate}

Let $ \{M(t); t\geq 0\} $ be the mean matrix semigroup
\[
M_{ij}(t) = \mathbb{E}\left(Y_{j}(t) \mid Y_{r}(0) = \delta_{ri},\ r=1,2,\ldots ,m\right),
\]
for $ i,j \in \{1,2,\ldots,m\} $ and where the $ \delta_{ij} $ are Kronecker deltas. For a vector $ \alpha \in \mathbb{R}^{m}$, let $ \mbox{diag}(\alpha)$ denote the $ m\times m$ diagonal matrix with $ \mbox{diag}(\alpha)_{ii} = \alpha_{i} $, $i=1,\ldots,m$. The mean matrix semigroup has the infinitesimal generator
\[ 
A(\pi,Q) = \mbox{diag}(\beta)\mbox{diag}(\pi) - \mbox{diag}(\gamma) + Q
\]
and $ M(t) = \exp (A(\pi,Q) t) $ for $ t \geq 0 $ \citep{Athreya:68}. By \citet[Theorem~2.7]{Seneta:81}, if $ Q $ is irreducible, then
\[
\exp\left[A(\pi,Q) t\right] = (1 + o(1))\exp\left[\tau(\pi,Q) t\right] {\bf w} {\bf v}^{T},
\]
elementwise as $t\to\infty$,
where $ {\bf w} $ and $ {\bf v}^{T} $ are the left and right eigenvectors of $ A(\pi,Q) $ corresponding to the dominant eigenvalue $ \tau(\pi,Q) $ and normed so that $ {\bf v}^{T} {\bf w} = 1 $. Since $ {\bf w} $ and $ {\bf v} $ are both strictly positive \citep[Theorem~2.6(b)]{Seneta:81}, $ \tau(\pi,Q) $ is the growth rate of the expected number of infected individuals during the initial stages of the epidemic.

The following result shows that, under a certain condition on the recovery and infection rates, there is an optimal distribution of susceptible individuals which minimises the growth rate. 

\begin{theorem} \label{thm1}
Let $ \chi(\beta,\gamma) = (1 - \sum_{j=1}^{m} \gamma_{j}/\beta_{j})(\sum_{j=1}^{m} \beta^{-1}_{j})^{-1} $ and define $ \pi^{\ast}_{i} = (\gamma_{i} + \chi(\beta,\gamma))/\beta_{i} $ for $ i = 1,\ldots,m $. If $ \gamma_{i} > -\chi(\beta,\gamma) $ for all $ i=1,\ldots,m $, then $ \pi^{\ast} \in \mathcal{S}$  and there exists a $ Q^{\ast} \in \mathcal{Q} $ such that 
\begin{equation}
\tau(\pi,Q^{\ast}) \geq \tau(\pi^{\ast},Q^{\ast}) = \tau(\pi^{\ast}, Q) = \chi(\beta,\gamma), \label{thm1:eq1}
\end{equation}
for all $ \pi \in \mathcal{S} $ and all $ Q \in \mathcal{Q} $. 
\end{theorem}

\begin{proof}
The final equality in (\ref{thm1:eq1}) holds as, for any $ Q \in \mathcal{Q}$, $ A(\pi^{\ast},Q) {\bf 1} =  \chi(\beta,\gamma) {\bf 1} $. From \citet[Corollary 3, page 52]{Seneta:81}, $ \tau(\pi^{\ast},Q) = \chi(\beta,\gamma) $. 

To prove the inequality in (\ref{thm1:eq1}), fix $ Q \in \mathcal{Q}$. As $ \mathcal{S} $ is a convex set, it follows from \citet[Theorem~4.1]{Friedland:81} that $ \tau(\cdot,Q) $ is a strictly convex functional on $ \mathcal{S} $ . Therefore, $ \hat{\pi}(Q) $ will minimise $ \tau(\cdot,Q) $ if and only if
\[
\sum_{i=1}^{m} \frac{\partial \tau(\pi,Q) }{\partial \pi_{i}} \Big|_{\pi = \hat{\pi}(Q)} \left(\pi_{i} - \hat{\pi}(Q)_{i} \right) \geq 0,
\]
for all $ \pi \in \mathcal{S}$. The partial derivatives of $ \tau(\cdot,Q) $ are
\[
 \frac{\partial \tau(\pi,Q) }{\partial \pi_{i}}  =  \frac{\beta_{i} w_{i} v_{i}}{{\bf w}^{T} {\bf v}},
\]
where $ {\bf w} $ and $ {\bf v} $ are the left and right eigenvectors of $ A(\pi,Q) $ corresponding to the dominant eigenvalue \citep[pg 183]{SS:90}.  As  $ \beta_{i} \pi^{\ast}_{i} - \gamma_{i} = \chi(\beta,\gamma) $ for $ i =1,\ldots,m $, the eigenvectors of $ A(\pi^{\ast},Q) $ and $ Q $ coincide so $ {\bf 1} $ is a right eigenvector of $ A(\pi^{\ast},Q)$ for any $ Q\in\mathcal{Q}$. Therefore, the inequality in (\ref{thm1:eq1}) will follow if there exists a  $ Q^{\ast} $ with left eigenvector $ {\bf w} $  such that
\begin{equation}
\sum_{i=1}^{m} \beta_{i} w_{i} \left(\pi_{i} - \pi^{\ast}_{i} \right) \geq 0, \label{thm1:eq2}
\end{equation}
for all $ \pi \in \mathcal{S}$. Set $ Q^{\ast}$ such that $ Q^{\ast}_{ij} = \beta_{j}^{-1} $ for $ i \neq j $. The left eigenvector of $ Q^{\ast}$ satisfies  $ w_{i} \propto \beta_{i}^{-1} $. Therefore, inequality (\ref{thm1:eq2}) holds which proves $ \tau(\pi,Q^{\ast}) \geq \tau(\pi^{\ast},Q^{\ast})  $ for all $ \pi \in \mathcal{S} $. 
\end{proof}

Under the condition of Theorem~\ref{thm1}, $ \pi^{\ast} $ is the distribution of susceptible individuals which minimises the expected growth rate of the epidemic when infected individuals move to maximise the expected growth rate of the epidemic. A corresponding optimal migration rate matrix for susceptibles can easily be determined, but the optimal migration rate matrix is not unique. For example, any transition rate matrix $ R $ satisfying the detailed balance equations $ R_{ij} \pi^{\ast}_{i} = R_{ji} \pi^{\ast}_{j} $, for all $ i,j \in\{1,\ldots,m\}$ has $ \pi^{\ast} $ as its equilibrium distribution. Many other constructions are possible. 

Theorem~\ref{thm1} excludes the boundary case where $ \gamma_{i} \geq -\chi(\beta,\gamma) $ for all $ i $ and $ \gamma_{j} = -\chi(\beta,\gamma) $ for some $ j $. The difficulty in this case is that $ \pi^{\ast} \not\in\mathcal{S} $ since $ \pi^{\ast}_{j} = 0 $ and there is no corresponding irreducible migration rate matrix for susceptibles. However, if we define, for any $ \epsilon > 0 $, $ \pi^{\epsilon} \in \mathcal{S} $ by $ \pi^{\epsilon}_{i} = (\gamma_{i} + \epsilon + \chi(\beta,\gamma))/\beta_{i} $ for all $ i $, then the calculations from the proof of Theorem~\ref{thm1} shows that $ (\pi^{\epsilon}, Q^{\ast}) $ is an $ \epsilon$-saddle point for $ \tau(\pi,Q) $.

The condition imposed in Theorem~\ref{thm1} will be satisfied provided the recovery rates do not vary too much between patches. In particular, if $ \gamma_{i} = \gamma $ for $ i=1,\ldots,m$ , then $ \chi(\beta,\gamma) = (\sum_{j=1}^{m} \beta_{j}^{-1})^{-1} - \gamma $, and the condition holds for any  $ \beta_{i},\ i=1,\ldots,m$. Also, if $ \chi(\beta,\gamma) \geq 0 $, then the condition must hold as the recovery rates are all nonnegative. On the other hand, when the condition does not hold, $ \chi(\beta,\gamma) < 0$. The following corollary shows that in this case there is a distribution of susceptible individuals for which the expected growth rate is negative, regardless of the movements of the infected individuals.

\begin{corollary} \label{cor1}
Suppose $ \chi(\beta,\gamma) < 0 $. Then $ \inf_{\pi \in \mathcal{S}} \sup_{Q\in\mathcal{Q}} \tau (\pi,Q) < 0 $. 
\end{corollary}

\begin{proof}
Define $ \tilde{\pi}_{i} = \frac{\gamma_{i}}{\beta_{i}}  (\sum_{j=1}^{m} \frac{\gamma_{j}}{\beta_{j}})^{-1} $. Then for all $ Q \in \mathcal{Q} $, $(A(\tilde{\pi},Q) {\bf 1})_{i} = \gamma_{i}((\sum_{j=1}^{m} \frac{\gamma_{j}}{\beta_{j}})^{-1} -1 ),\ i=1,\ldots,m$. Now $ \chi(\beta,\gamma) < 0 $ implies $ (\sum_{j=1}^{m} \frac{\gamma_{j}}{\beta_{j}})^{-1} < 1 $. By \citet[Corollary 3, page 52]{Seneta:81}, $ \tau(\tilde{\pi},Q) < 0 $ for all $ Q \in \mathcal{Q}$. 
\end{proof}

\subsection{Optimising $R_0$, minor outbreak probability, and expected total size}

When the condition of Theorem~\ref{thm1} does not hold, an alternative approach to controlling the disease spread is needed. As noted in \citet[Section 4.1]{Clancy:96}, the total size of the branching process approximating the epidemic is the same as the total of an embedded Galton-Watson process. The behaviour of this Galton-Watson process is largely determined by the expected number of infectives produced by a single infective before its recovery. Denote by $ \Lambda_{ij}(\pi,Q) $ the expected number of infectives produced in group~$ j $ by an individual first infected in group~$ i $. Then from \citet[Section 4.1]{Clancy:96}, $ \Lambda(\pi,Q) =  L(Q) \mbox{diag}(\beta) \mbox{diag}(\pi) $, where $ L_{ij}(Q) $ is the expected amount of time that an individual who is first infected while in group~$ i $ spends in group~$ j $ before recovery. By \citet[Proposition 2]{PS:02}, $ L(Q) =  (\mbox{diag}(\gamma) - Q)^{-1} $. Therefore, 
\[
\Lambda(\pi,Q) =  \left(\mbox{diag}(\gamma) - Q \right)^{-1} \mbox{diag}(\beta) \mbox{diag}(\pi).
\]
The basic reproduction rate is the spectral radius of $ \Lambda(\pi,Q)$, which is denoted by $ R_{0}(\pi,Q) $. It is known that if $ R_{0}(\pi,Q) \leq 1 $, then the Galton-Watson process goes extinct in finite time with probability one. Minimising $ R_{0}(\pi,Q)$ provides an alternate means of limiting the spread of the disease.

\begin{theorem} \label{thm2}
Let $ \omega(\beta,\gamma) = (\sum_{j=1}^{m} \frac{\gamma_{j}}{\beta_{j}})^{-1} $ and define $ \tilde{\pi}_{i} = \frac{\gamma_{i}}{\beta_{i}} \omega(\beta,\gamma)  $ for $ i = 1,\ldots,m $. There exists a $ \tilde{Q} \in \mathcal{Q} $ such that 
\begin{equation}
R_{0}(\pi,\tilde{Q}) \geq R_{0}(\tilde{\pi},\tilde{Q}) = R_{0}(\tilde{\pi}, Q) = \omega(\beta,\gamma) , \label{thm2:eq1}
\end{equation}
for all $ \pi \in \mathcal{S} $ and all $ Q \in \mathcal{Q} $.
\end{theorem}

\begin{proof} 
The proof is similar to the proof of Theorem~\ref{thm1}. We first prove the final equality in (\ref{thm2:eq1}). For any $ Q \in \mathcal{Q} $, $ Q {\bf 1} = {\bf 0 }$ so $\left( \mbox{diag}(\gamma) - Q\right)^{-1} \gamma = {\bf 1} $. As
\begin{equation}
\Lambda(\tilde{\pi},Q) {\bf 1} =  \omega(\beta,\gamma)\left( \mbox{diag}(\gamma) - Q \right)^{-1} \mbox{diag}(\gamma) {\bf 1}  = \omega(\beta,\gamma)  {\bf 1}, \label{thm2:eq0}
\end{equation}
it follows that $ R_{0}(\tilde{\pi},Q) = \omega(\beta,\gamma) $ \citep[Theorem 1.6]{Seneta:81}.

% Details: 
%As $ Q {\bf 1} = 0 $ for all $ Q \in \mathcal{Q}^{0}$,
%\begin{eqnarray*}
%\gamma & = & \gamma - Q {\bf 1} \\
%& = & \left( \mbox{diag}(\gamma) - Q\right) {\bf 1} \\
%\left( \mbox{diag}(\gamma) - Q\right)^{-1} \gamma & = & {\bf 1} \\
%\left( \mbox{diag}(\gamma) - Q\right)^{-1} \mbox{diag}(\gamma) {\bf 1} & = & {\bf 1} \\
%\left( \mbox{diag}(\gamma) - Q\right)^{-1}\mbox{diag}(\beta)\mbox{diag}(\pi^{\ast}) {\bf 1}& = & \left(\sum_{j=1}^{m} \frac{\gamma_{j}}{\beta_{j}}\right)^{-1}{\bf 1}
%\end{eqnarray*}

To prove the inequality in (\ref{thm2:eq1}), fix $ Q \in \mathcal{Q} $. As $ \mathcal{S} $ is a convex set, \citet[Theorem~4.3]{Friedland:81} shows that $ R_{0}(\cdot,Q) $ is a strictly convex functional on $ \mathcal{S} $. Therefore, $ \hat{\pi}(Q) $ minimises $ R_{0}(\cdot,Q) $ if and only if
\begin{equation}
\sum_{i=1}^{m} \frac{\partial R_{0}(\pi,Q) }{\partial \pi_{i}} \Big|_{\pi = \hat{\pi}(Q)} \left(\pi_{i} - \hat{\pi}(Q)_{i} \right) \geq 0, \label{thm2:eq3}
\end{equation}
for all $ \pi \in \mathcal{S} $. The partial derivatives of $ R_{0}(\cdot,Q) $ are
\begin{equation}
\frac{\partial R_{0}(\pi,Q) }{\partial \pi_{i}}  =  \frac{\beta_{i} ({\bf w}   \left(\mbox{diag}(\gamma) - Q \right)^{-1})_{i}  {\bf v}_{i}}{{\bf w}^{T} {\bf v}}, \label{thm2:eq4}
\end{equation}
where $ {\bf w} $ and $ {\bf v} $ are the left and right eigenvectors of $ \Lambda(\pi,Q) $ corresponding to the dominant eigenvalue \citep[pg 183]{SS:90}.
As noted previously, $ {\bf 1} $ is a right eigenvector of  $\Lambda(\tilde{\pi},Q) $ corresponding to the dominant eigenvalue $ \omega(\beta,\gamma) $. Substituting $ \pi = \tilde{\pi} $ and $ {\bf v} = {\bf 1} $  in equation (\ref{thm2:eq4}) and combining with equation (\ref{thm2:eq3}), we see that the inequality in (\ref{thm2:eq1}) will follow if there exists a  $ \tilde{Q} \in \mathcal{Q} $ such that for all $ \pi \in \mathcal{S}$
\begin{equation}
\sum_{i=1}^{m} \beta_{i}  ({\bf w}   (\mbox{diag}(\gamma) - \tilde{Q} )^{-1})_{i} \left(\pi_{i} - \tilde{\pi}_{i} \right) \geq 0. \label{thm2:eq2}
\end{equation}
As in Theorem~\ref{thm1},  set $ \tilde{Q} $ such that $ \tilde{Q}_{ij} = \beta_{j}^{-1} $ for $ i \neq j $. Then the left eigenvector of $ \Lambda(\tilde{\pi},\tilde{Q}) $ corresponding to the dominant eigenvalue $ \omega(\beta,\gamma) $ satisfies $ w_{i} \propto \gamma_{i}/\beta_{i}$ so $ ({\bf w} (\mbox{diag}(\gamma) - \tilde{Q})^{-1})_{i} \propto \beta_{i}^{-1} $.  
% Details:
%\begin{eqnarray*}
%(\frac{1}{\beta})^{T} \tilde{Q} & = & 0 \\
%\omega(\beta,\gamma) (\frac{\gamma}{\beta})^{T}  & = & \omega(\beta,\gamma) (\frac{\gamma}{\beta})^{T}  - \omega(\beta,\gamma)(\frac{1}{\beta})^{T} \tilde{Q} \\
%& = & \omega(\beta,\gamma) (\frac{1}{\beta})^{T} ( \mbox{diag}(\gamma) - \tilde{Q}) \\
%\omega(\beta,\gamma) (\frac{\gamma}{\beta})^{T} ( \mbox{diag}(\gamma) - \tilde{Q})^{-1}  & = & \omega(\beta,\gamma) (\frac{1}{\beta})^{T} \\
%\omega(\beta,\gamma) (\frac{\gamma}{\beta})^{T} ( \mbox{diag}(\gamma) - \tilde{Q})^{-1} \mbox{diag}(\gamma) & = & \omega(\beta,\gamma) (\frac{1}{\beta})^{T}  \mbox{diag}(\gamma) \\
% (\frac{\gamma}{\beta})^{T} ( \mbox{diag}(\gamma) - \tilde{Q})^{-1} \mbox{diag}(\beta) \mbox{diag}(\pi) & = & \omega(\beta,\gamma) (\frac{\gamma}{\beta})^{T} 
%\end{eqnarray*}
Therefore, inequality (\ref{thm2:eq2}) holds which proves $ R_{0}(\pi,\tilde{Q}) \geq R_{0}(\tilde{\pi},\tilde{Q})  $ for all $ \pi \in \mathcal{S} $.
\end{proof}

Theorem~\ref{thm2} shows that $ \tilde{\pi} $ is the distribution of susceptible individuals which minimises the basic reproduction rate of the epidemic when infected individuals move to maximise the basic reproduction rate of the epidemic.

Although Theorems~\ref{thm1} and~\ref{thm2} give two different strategies for minimising the disease spread, the quantities $ \chi(\beta,\gamma) $ and $ \omega(\beta,\gamma) $ are closely related. First, $ \chi(\beta,\gamma) = 0 $ if and only if $ \omega(\beta,\gamma) = 1 $ in which case $ \pi^{\ast} = \tilde{\pi} $. Note we also have $ \pi^{\ast} = \tilde{\pi} $ when the recovery rates do not depend on the group. Second, $ \chi(\beta,\gamma) < 0 $ if and only if $ \omega(\beta,\gamma) < 1 $. Therefore, Theorems~\ref{thm1} and~\ref{thm2} and Corollary~\ref{cor1} imply that the $ R_{0} $-optimal strategy yields $ R_{0} > 1 $ if  and only if the $ \tau $-optimal strategy yields $ \tau > 0 $.

Taking $ R_{0} $ as the objective function has the advantage that it is always possible to give an explicit optimal strategy. Although $ R_{0} $ is more tractable than $ \tau $, it is only useful if the resulting optimal strategy reduces the extent of the original epidemic in some sense. In the context of ODE models, \citet{DHR:10} notes that epidemics with high $ R_{0} $ do not necessarily have a fast increase of incidence. Therefore, one might question the relevance of reducing $ R_{0} $ if the threshold cannot be achieved and, if the threshold can be achieved, the advantage of reducing $ R_{0} $ further. To see why it is always useful to reduce $ R_{0} $, it is necessary to consider the two cases $ \omega(\beta,\gamma) < 1 $ and $ \omega(\beta,\gamma) > 1 $ separately. The case where $ \omega(\beta,\gamma) = 1 $ is not considered since in that case $ \tilde{\pi} = \pi^{\ast} $.

We first examine how the optimal strategy from Theorem~\ref{thm2} relates to the probability that the branching process goes extinct in
finite time. This probability is determined by the smallest fixed point of the probability generating function for the offspring distribution. For the branching process determined by (\ref{Model:Eq1})--(\ref{Model:Eq3}), this probability generating function is
$$
g_{i}({\bf u};\pi,Q) = \frac{\sum_{j\neq i} Q_{ij} u_{j} + \beta_{i} \pi_{i} u_{i}^{2} + \gamma_{i}}{\sum_{j\neq i} Q_{ij}  + \beta_{i} \pi_{i}  + \gamma_{i}}.
$$
The function $ {\bf g}({\bf u};\pi,Q) = (g_{1}({\bf u};\pi,Q),\ldots,g_{m}({\bf u};\pi,Q)) $ always has a fixed point at $ {\bf 1} $, that is $ {\bf g}({\bf 1};\pi,Q) ={\bf  1} $. If $ \tau(\pi,Q) > 0 $, then $ {\bf g}(\cdot;\pi,Q) $ has a second fixed point, which is unique in $ (0,1)^{m} $ \citep[Section 2.3]{AD:13}. Denote the smallest fixed point of $ {\bf g}(\cdot;\pi,Q) $ in $[0,1]^{m}$ by $ {\bf q}(\pi,Q) $. The probability of extinction in finite time is given by
$$
\lim_{t\rightarrow\infty} \mathbb{P}\left(Y_{j}(t) = 0 \mbox{ for } j =1,\ldots,m \mid Y_{i}(0) = y_{i} \mbox{ for } i=1,\ldots,m \right) = \prod_{i=1}^{m} q_{i}^{y_{i}}(\pi,Q).
$$
The following results shows that taking the distribution of susceptibles to be $ \tilde{\pi} $ maximises the probability of extinction in finite time minimised over the starting location of the initial infected individual. 

\begin{theorem} \label{thm3}
If $ \omega(\beta,\gamma) > 1 $, then for any $ \epsilon > 0 $ there exists a $ Q^{\epsilon} \in \mathcal{Q} $ such that 
\begin{equation}
\min_{i}\left\{q_{i}(\pi,Q^{\epsilon})\right\} - \epsilon \leq \min_{i}\left\{q_{i}(\tilde{\pi},Q^{\epsilon})\right\} = \min_{i}\left\{q_{i}(\tilde{\pi},Q)\right\} = \omega(\beta,\gamma)^{-1} \label{thm3:eq1}
\end{equation}
for all $ \pi \in \mathcal{S} $ and $ Q \in \mathcal{Q} $. 
\end{theorem}

\begin{proof}
When $ \pi = \tilde{\pi} $, the probability generating function of the offspring distribution is 
$$
g_{i}({\bf u};\tilde{\pi},Q) = \frac{\sum_{j\neq i} Q_{ij} u_{j} + \gamma_{i} \omega(\beta,\gamma) u_{i}^{2} + \gamma_{i}}{\sum_{j\neq i} Q_{ij}  + \gamma_{i} \omega(\beta,\gamma)  + \gamma_{i}}.
$$
It can be verified by substitution that $ q_{i}(\tilde{\pi},Q) =
\omega(\beta,\gamma)^{-1},\ i=1,\ldots,m $, for all $ Q \in \mathcal{Q}$. It remains to prove the inequality in (\ref{thm3:eq1}), which we achieve by determining an upper bound on $ {\bf q}(\pi,Q) $ for certain $ Q $.

As $ g_{i}(\cdot;\pi,Q)  $ is a monotone function and $ {\bf q}(\pi,Q) $ is a fixed point of $ {\bf g}(\cdot;\pi,Q) $, it follows that if, for some $ {\bf p} \in (0,1)^{m} $, $ g_{i}({\bf p};\pi,Q) \leq p_{i},\ i=1,\ldots,m $, then $ q_{i}(\pi,Q) \leq p_{i},\ i=1,\ldots,m $. Let $ \delta_{i} = \sum_{j\neq i} Q^{\epsilon}_{ij} $ and choose $ Q^{\epsilon} $ such that $ \sum_{i=1}^{m} \delta_{i} /\beta_{i} < \epsilon $. For any $ {\bf p} \in (0,1)^{m} $,
\begin{equation}
g_{i}({\bf p};\pi,Q^{\epsilon}) \leq \frac{\delta_{i} + \beta_{i}\pi_{i} p_{i}^{2} + \gamma_{i}}{\delta_{i} + \beta_{i}\pi_{i} + \gamma_{i}} = p_{i} + \frac{\delta_{i} +  \beta_{i}\pi_{i} p_{i}^{2} + \gamma_{i} - (\beta_{i}\pi_{i} + \gamma_{i} + \delta_{i})p_{i}}{\delta_{i} + \beta_{i}\pi_{i}  + \gamma_{i}}. \label{thm3:eq2}
\end{equation}
It can be verified by substitution into (\ref{thm3:eq2}) that 
$$ 
p_{i} = \left( \frac{\gamma_{i} + \delta_{i}}{\beta_{i}\pi_{i}} \wedge 1\right)
$$
is an upper bound on $ {\bf q}(\pi,Q^{\epsilon}) $. Therefore,
$$
\min_{i}  q_{i}(\pi,Q^{\epsilon}) - \epsilon \leq \min_{i} \left\{\frac{\gamma_{i} + \delta_{i}}{\beta_{i} \pi_{i}}\right\} - \epsilon.
$$
Suppose that, for some $ \pi $ and all $ i =1,\ldots,m$,
$$
\frac{\gamma_{i} + \delta_{i}}{\beta_{i}\pi_{i}} - \epsilon > \omega(\beta,\gamma)^{-1},
$$
then
\begin{equation}
\pi_{i} < \tilde{\pi_{i}} + \frac{\delta_{i} \omega(\beta,\gamma) }{\beta_{i}} - \epsilon \pi_{i} \omega(\beta,\gamma). \label{thm3:eq3}
\end{equation}
By summing over $ i $ in inequality (\ref{thm3:eq3}), we arrive at the
contradiction $ \sum_{i=1}^{m} \delta_{i} /\beta_{i} > \epsilon $.
Therefore, for all $ \pi \in \mathcal{S} $, there is at least one $ i $ such that 
$$
\frac{\gamma_{i} + \delta_{i}}{\beta_{i}\pi_{i}} - \epsilon \leq \omega(\beta,\gamma)^{-1},
$$
which proves the inequality in (\ref{thm3:eq1}).
\end{proof}

The previous theorem provides support for minimising $ R_{0} $ when $ \omega(\beta,\gamma) > 1 $; it remains to justify minimising $ R_{0} $ when $ \omega(\beta,\gamma) < 1 $. Let $ T_{ij}(\pi,Q) $ denote the number of individuals infected in node $ j $ starting from a single infected individual at node $ i $. We now consider the effect of migration on the expected total size of the epidemic,
$$
\max_{i}  \mathbb{E} \left(\sum_{j=1}^{m} T_{ij}(\pi,Q) \right).
$$
The next result shows that taking the distribution of susceptibles to be $ \tilde{\pi} $ minimises the total size of the epidemic maximised over the starting location of the initial infected individual.

\begin{theorem} \label{thm4}
If $ \omega(\beta,\gamma) < 1 $, then for any $ \epsilon > 0 $ there exists a $ Q^{\epsilon} \in \mathcal{Q} $ such that 
\begin{eqnarray}
\max_{i}  \mathbb{E} \left(\sum_{j=1}^{m} T_{ij}(\pi,Q^{\epsilon}) \right) + \epsilon & \geq & \max_{i}  \mathbb{E} \left(\sum_{j=1}^{m}  T_{ij}(\tilde{\pi},Q^{\epsilon}) \right) \label{Thm4:Eq1}\\
 & = &  \max_{i}  \mathbb{E} \left(\sum_{j=1}^{m}  T_{ij}(\tilde{\pi},Q) \right) 
 =  (1 - \omega(\beta,\gamma))^{-1} \nonumber
\end{eqnarray}
for all $ \pi \in \mathcal{S} $ and $ Q \in \mathcal{Q} $.
\end{theorem}

\begin{proof}

As the total size of the branching process approximating the epidemic is the same as that of an embedded Galton-Watson process whose offspring distribution has mean matrix $ \Lambda(\pi,Q) $ \citep[Section 4.1]{Clancy:96}, 
$$
\mathbb{E} \left(\sum_{j=1}^{m} T_{ij}(\pi,Q) \right)   = \left(\sum_{r=0}^{\infty} \Lambda(\pi,Q)^{r}\right)_{i},
$$
which is finite if and only if the spectral radius of $ \Lambda(\pi,Q) $ is strictly less than one. From equation~(\ref{thm2:eq0}), 
$$
\mathbb{E} \left(\sum_{j=1}^{m}  T_{ij}(\tilde{\pi},Q) \right)   = (1 - \omega(\beta,\gamma))^{-1},
$$
for all $i =1,\ldots, m$ and all $ Q \in \mathcal{Q} $. This proves the equality in (\ref{Thm4:Eq1}). To complete the proof, it remains to show that for any $ \epsilon > 0 $ there exists a $ Q^{\epsilon} \in \mathcal{Q} $, such that 
\begin{equation}
\max_{i} \mathbb{E} \left(\sum_{j=1}^{m} T_{ij}(\pi,Q^{\epsilon}) \right) + \epsilon \geq  (1 - \omega(\beta,\gamma))^{-1}, \label{Thm4:Eq2}
\end{equation}
for all $ \pi \in \mathcal{S}$. Take $ Q^{\epsilon} = \delta ({\bf 1} {\bf 1}^{T} - m I) $ where $ \delta $ satisfies
\begin{equation}
 0 < \delta m \sum_{i=1}^{m} \gamma_{i}^{-1} < \epsilon(1-\omega(\beta,\gamma)) (\omega(\beta,\gamma)^{-1} - 1). \label{Thm4:Eq4}
\end{equation}
Applying the Woodbury matrix identity to $ \Lambda(\pi,Q^{\epsilon}) $, we obtain
$$
\Lambda(\pi,Q^{\epsilon}) = \left( I + \frac{\delta}{(1 - \delta {\bf 1}^{T} \Gamma_{\delta}^{-1} {\bf 1})} {\bf 1}{\bf 1}^{T} \Gamma_{\delta}^{-1}\right)  \Gamma_{\delta}^{-1} \mbox{diag}(\beta) \mbox{diag}(\pi),
$$
where $ \Gamma_{\delta} = \mbox{diag}(\gamma + \delta m {\bf 1}) $. Hence, in the partial order of positive matrices,
$$
\mbox{diag}(\beta) \mbox{diag}(\pi)  \Gamma_{\delta}^{-1}  \leq \Lambda(\pi,Q^{\epsilon}).
$$
Therefore, the expected total size is finite for all $ \epsilon > 0 $ only if $ \beta_{i} \pi_{i} /\gamma_{i} < 1,\ i=1,\ldots,m$, in which case
\begin{equation}
\left(1 - \frac{\beta_{i}\pi_{i}}{\gamma_{i} + \delta m}\right)^{-1} \leq  \mathbb{E} \left(\sum_{j=1}^{m}  T_{ij}(\pi,Q) \right), \label{Thm4:Eq3}
\end{equation}
for all $ i =1,\ldots,m$. Suppose that inequality (\ref{Thm4:Eq2}) did not hold for some $ \pi \in \mathcal{S} $. Then, for all $ i $,
$$
\mathbb{E} \left(\sum_{j=1}^{m} T_{ij}(\pi,Q^{\epsilon}) \right) + \epsilon <  (1 - \omega(\beta,\gamma))^{-1}.
$$
Inequality (\ref{Thm4:Eq3}) would then implies
$$
\left(1 - \frac{\beta_{i}\pi_{i}}{\gamma_{i} + \delta m}\right)^{-1} + \epsilon <  (1 - \omega(\beta,\gamma))^{-1}.
$$
This inequality can be rearranged to
$$
\tilde{\pi}_{i} - \epsilon (1-\omega(\beta,\gamma)) \frac{\gamma_{i}}{\beta_{i}} \left( 1 - \frac{\beta_{i}\pi_{i}}{\gamma_{i} + \delta m} \right) > \left(1 - \frac{\delta m}{\gamma_{i} + \delta m}\right) \pi_{i}.
$$
Summing this inequality over $ i $, we find
$$
\epsilon (1-\omega(\beta,\gamma)) \sum_{i=1}^{m} \frac{\gamma_{i}}{\beta_{i}} \left( 1 - \frac{\beta_{i}\pi_{i}}{\gamma_{i} + \delta m} \right) < \delta m \sum_{i=1}^{m} \frac{\pi_{i}}{\gamma_{i} + \delta m}.
$$
As $\delta $ is chosen to satisfy the inequality (\ref{Thm4:Eq4}), we obtain a contradiction. Hence, inequality (\ref{Thm4:Eq2}) holds for all $ \pi \in \mathcal{S} $. 
\end{proof}

\section{Numerical  comparisons}

In this section we investigate numerically two issues. The first issue concerns the optimal distribution of susceptibles with respect to minimising the expected growth rate. Theorem~\ref{thm1} gives the optimal distribution only if $ \gamma_{i} > -\chi(\beta,\gamma) $ for all $ i $. We have also seen that  if $ \gamma_{i} \geq -\chi(\beta,\gamma) $ for all $ i $ and $ \gamma_{j} = -\chi(\beta,\gamma) $ for some $ j $, then there is a sequence of distributions $ \pi^{\epsilon} $ which achieve within $ \epsilon $ the optimal value of $ \tau $ and for which $ \lim_{\epsilon\rightarrow 0} \pi^{\epsilon}_{j} = 0 $. We might expect that if the recovery rate for this group were to decrease, then the optimal distribution would place no susceptibles in group $ j $. 
This was investigated in a two group epidemic with $ \beta_{1} = 1 $, $ \beta_{2} = 2 $, and $ \gamma_{1} $ and $ \gamma_{2} $ in $ (0,4) $. For these epidemics both $\sup_{Q\in\mathcal{Q}} \inf_{\pi\in\mathcal{S}} \tau(\pi,Q) $ and $ \inf_{\pi\in\mathcal{S}} \sup_{Q\in\mathcal{Q}} \tau(\pi,Q) $ were computed by nested optimisation using the \emph{optim} and \emph{optimize} functions in \textsc{R} \citep{R:2011}. The two quantities differed by less than $10^{-4}$ in all instances computed. The optimal value of $ \tau(\pi,Q) $ is plotted in Figure \ref{Fig1}. Note that in most of the region plotted growth rate is negative. This is to be expected as when the condition of Theorem~\ref{thm1} does not hold, the optimal growth rate must be negative from Corollary~\ref{cor1}. The numerical results confirms our intuition that $ \pi^{\ast}_{i} = 0 $ if $ \gamma_{i} \leq -\chi(\beta,\gamma) $ for the two group model. From the plot it is seen that if $ \gamma_{1} < - \chi(\beta,\gamma)  \leq \gamma_{2} $, then increasing $ \gamma_{2} $ has no effect on the optimal value of $ \tau(\pi,Q) $. This is explained as when $ \gamma_{1} < - \chi(\beta,\gamma)  \leq \gamma_{2} $, the optimal distribution of susceptibles has $ \pi^{\ast}_{2} = 1$. On the other hand, the inequality $ \gamma_{1} < - \chi(\beta,\gamma) $ implies $ \gamma_{1} < \gamma_{2} - \beta_{2} $, so the infectives slow the decrease of the epidemic by moving to group one. Therefore, increasing the recovery rate in group two has no effect on the growth rate of the epidemic and  $ \inf_{\pi\in\mathcal{S}} \sup_{Q\in\mathcal{Q}}\tau(\pi,Q) = - \gamma_{1} $.

\begin{figure}
\begin{center}
\includegraphics[height=10cm]{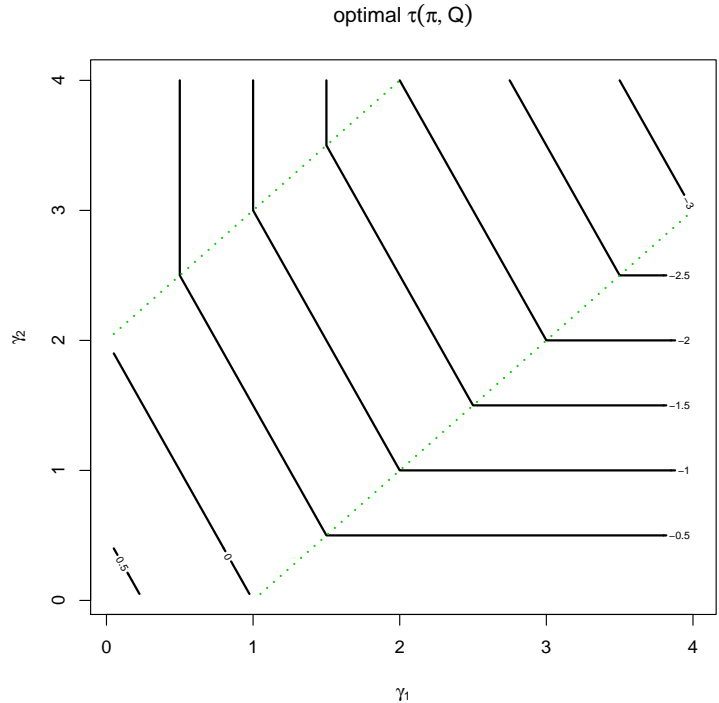}
\end{center}
\caption{A contour plot of the optimal value of $ \tau(\pi,Q) $ for the two group model with $ \beta_{1} = 1 $ and $ \beta_{2} = 2 $. The region between the dotted lines corresponds to the region where $ \gamma_{i} \geq - \chi(\beta,\gamma) $ for $ i=1,2$. Note that where the contour lines are horizontal, $ \tau(\pi,Q) = - \gamma_{2} $ and where the contour lines are vertical $ \tau(\pi,Q) = - \gamma_{1} $.} \label{Fig1}
\end{figure}

By construction, $ \pi^{\ast} $ and $ \tilde{\pi} $ are the optimal distribution of susceptibles for minimising  $ \tau $ and $ R_{0}$ respectively. We now consider their performance on the alternate criteria, that is we calculate $ \sup_{Q \in \mathcal{Q}} R_{0}(\pi^{\ast},Q) $ and $ \sup_{Q\in\mathcal{Q}} \tau(\tilde{\pi},Q) $. Figures \ref{Fig2} and \ref{Fig3} show how much these quantities are
increased by taking alternate optimal distributions of susceptibles. Qualitatively, both figures are very similar. Both quantities plotted achieve their minimum for the same set of $ \gamma $, indicated by the dashed line, as $ \pi^{\ast} = \tilde{\pi} $ for these values of $ \gamma $. Also, an abrupt change in the contours occur along the dashed line in both figures. This is due to $ \pi^{\ast} $ placing zero probability in one of the groups for those values of $ \gamma $ outside the dashed lines. 

For most values of $ \gamma $ the optimal choice for one criterion appears to result in reasonable performance in the other. In particular,  in the region where $ \gamma_{1} \approx \gamma_{2} \approx \beta_{1}\beta_{2}/(\beta_{1} + \beta_{2}) $, $ \tau $  and $ R_{0} $ take approximately the same value under $ \pi^{\ast} $ and $ \tilde{\pi} $. However, for small values of $ \gamma $, the performance of the alternate distributions rapidly deteriorates for both $ \tau $ and $ R_{0} $. This is expected as when $ \gamma $ is small, $ \chi(\beta,\gamma) $ tends to be large which causes the difference between $ \pi^{\ast} $ and $ \tilde{\pi} $ to also be large.

\begin{figure}
\begin{center}
\includegraphics[height=10cm]{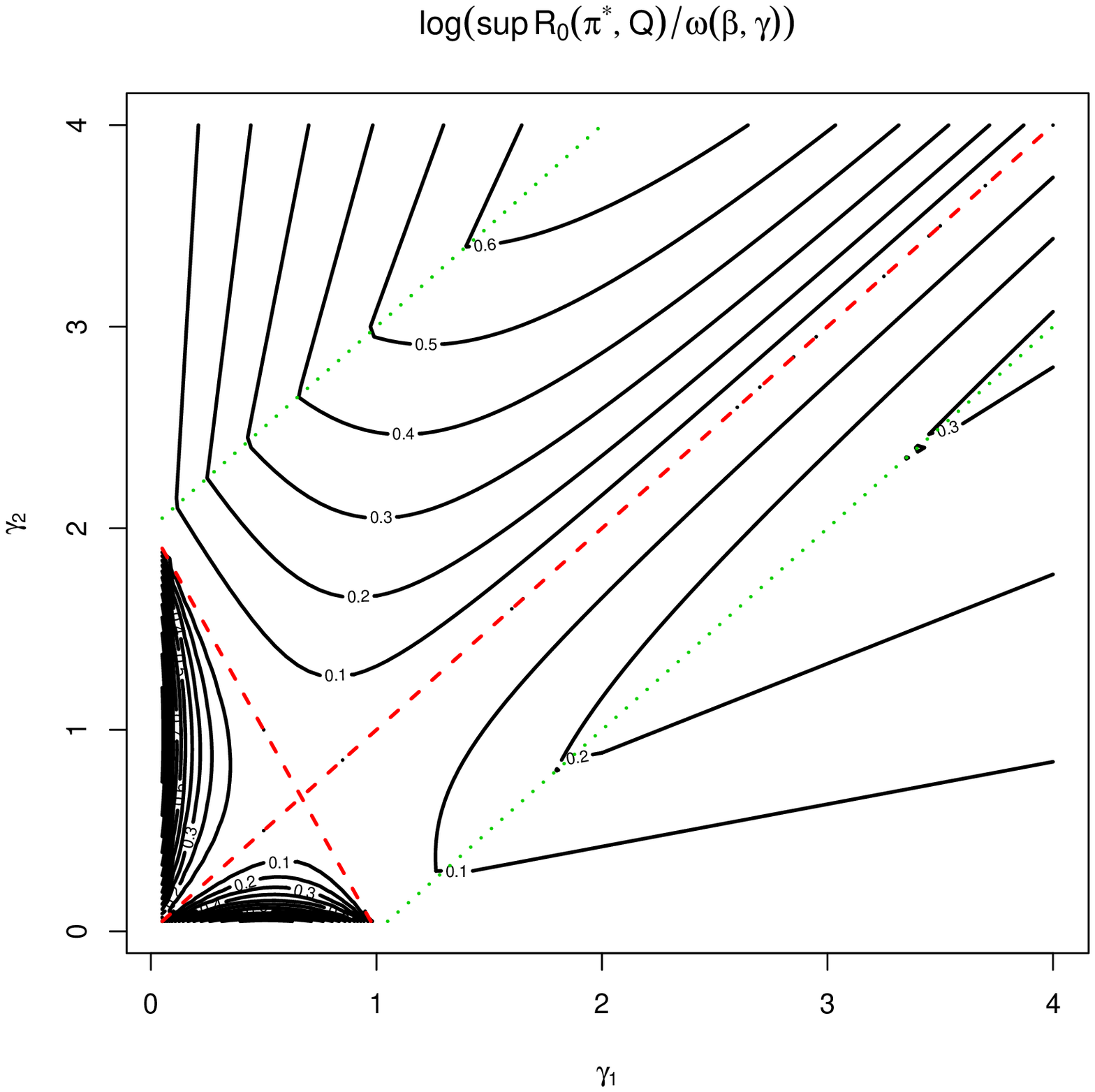}
\end{center}
\caption{A contour plot of the optimal value of $ \log(\sup_{Q\in\mathcal{Q}} R_{0}(\pi^{\ast},Q) / \omega(\beta,\gamma)) $ for the two group model with $ \beta_{1} = 1 $ and $ \beta_{2} = 2 $. The dashed lines correspond to where $ \omega(\beta,\gamma) = 1 $ so $ \pi^{\ast} = \tilde{\pi} $. The region between the dotted lines corresponds to the region where $ \gamma_{i} \geq - \chi(\beta,\gamma) $ for $ i=1,2$. } \label{Fig2}
\end{figure}

\begin{figure}
\begin{center}
\includegraphics[height=10cm]{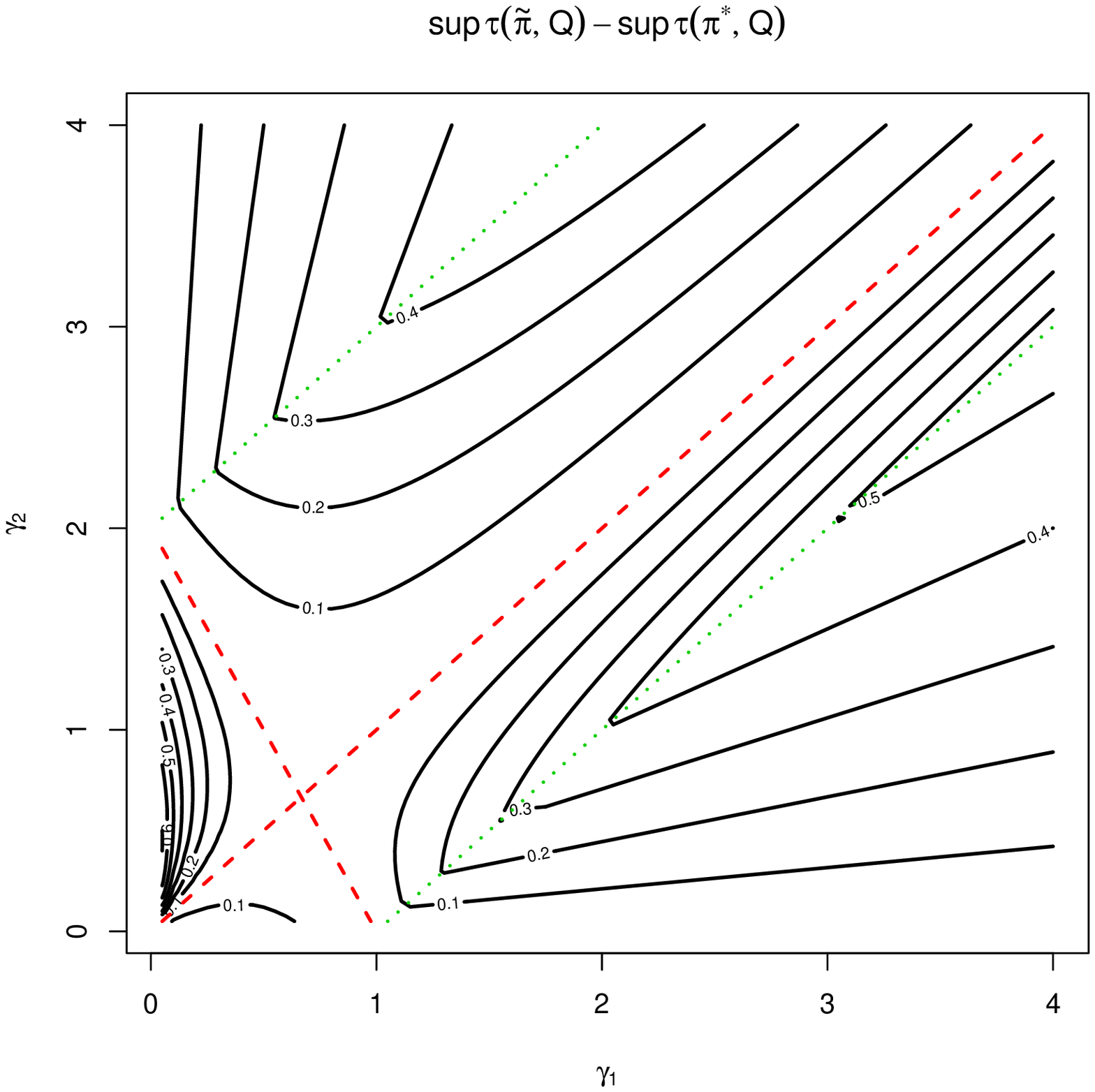}
\end{center}
\caption{A contour plot of the optimal value of $ \sup_{Q\in\mathcal{Q}} \tau(\tilde{\pi},Q)  - \sup_{Q\in\mathcal{Q}} \tau(\pi^{\ast},Q) $ for the two group model with $ \beta_{1} = 1 $ and $ \beta_{2} = 2 $. The dashed lines correspond to where $ \chi(\beta,\gamma) = 0 $ so $ \pi^{\ast} = \tilde{\pi} $. The region between the dotted lines corresponds to the region where $ \gamma_{i} \geq - \chi(\beta,\gamma) $ for $ i=1,2$. } \label{Fig3}
\end{figure}

\section{Discussion}

The conclusions of Theorems~\ref{thm1} and~\ref{thm2} are in part not surprising; in order to minimise the spread of the disease most susceptible individuals should belong to groups with relatively low infection rates and high recovery rates. However, for the form of contact rate assumed here, this needs to be balanced with the fact that contact rates are higher in groups with larger populations. Although Theorems~\ref{thm3} and~\ref{thm4} showed that this balance is achieved in the same way for $ R_{0}$-optimal, extinction probability optimal, and expected total size optimal distributions of susceptibles, it was achieved differently for $ \tau$-optimal distribution of susceptibles. It is conceivable that this balance might be achieved differently for other measures of disease spread.

In our analysis, we have focussed on the branching process approximation to the epidemic. Another widely used approximation is provided by the solution to an ordinary differential equation (ODE). Assume that infected individuals recover without immunity. For the epidemic described at the beginning of Section \ref{Sec:Main}, the ODE approximation is given by the solution to
\begin{eqnarray}
\frac{dx_{i}}{dt} & = & \sum_{j\neq i} R_{ji} x_{j}(t) + R_{ii} x_{i}(t) + \gamma_{i} y_{i}(t) - \beta_{i} x_{i}(t) y_{i}(t) \label{Dis:Eq0}\\ 
\frac{dy_{i}}{dt} & = & \sum_{j\neq i} Q_{ji} y_{j}(t) + Q_{ii} y_{i}(t) - \gamma_{i} y_{i}(t) + \beta_{i} x_{i}(t) y_{i}(t). \label{Dis:Eq1}
\end{eqnarray}
It is known that if $  N^{-1} X_{i}(0) \stackrel{p}{\rightarrow} x_{i}(0) $ and $ N^{-1} Y_{i}(0) \stackrel{p}{\rightarrow} y_{i}(0) $  for $i=1,\ldots,m$ as $ N \rightarrow \infty$, then for any finite $ T>0 $ and any $ \epsilon > 0$
$$
\lim_{N\rightarrow\infty} \Pr \left( \sup_{t \in [0,T]} \left( \sum_{i=1}^{m} |N^{-1} X_{i}(t) - x_{i}(t)| + \sum_{i=1}^{m} | N^{-1}Y_{i}(t) - y_{i}(t)|\right) > \epsilon\right) = 0,
$$
\citep[see][]{Kurtz:70,DN:08}.

Theorems~\ref{thm1} and~\ref{thm2} can still be used to determine the optimal distribution of susceptibles for the ODE model (\ref{Dis:Eq0}) - (\ref{Dis:Eq1}). First, consider the application of Theorem~\ref{thm1}. The spectrum of the Jacobian of the ODE model at the disease free equilibrium is given by the union of the spectrum of $ A(\pi,Q) $, where $ \pi $ is the unique solution to $ \pi R = 0 $ subject to $ \pi {\bf 1} = 1 $, and the spectrum of $ R $ with the zero eigenvalue removed. Therefore, if $ \gamma_{i} \geq -\chi(\beta,\gamma)$ for $i=1,\ldots,m$,  then Theorem~\ref{thm1} determines the $ \tau $ optimal choice of $ \pi $. However, for this to be attained, $ R $ must be chosen so that the non-zero eigenvalues of $ R $ have real part less than $ \chi(\beta,\gamma) $. Theorem~\ref{thm2} can similarly be applied to the ODE model. The next generation matrix \citep[Section 3]{vdDM:02} for the ODE model is given by $ \Lambda(\pi,Q)^{T} $. As the basic reproduction number for the ODE model is given by the spectral radius of the next generation matrix, Theorem~\ref{thm2} determines the $ R_{0}$ optimal distribution of susceptibles in the metapopulation. We are unaware of an interpretation of Theorems~\ref{thm3} and~\ref{thm4} for the ODE model.

We have previously noted that the desire for susceptible individuals to belong to a group with a low infection rate and high recovery rate needs to be balanced with the fact that contact rates are higher in groups with larger populations. This was due to the assumption of density dependent contact rates. An alternative is to assume frequency dependent contact rates, that is to assume the per capita contact rate in a group does not depend on the size of the group. \citet{ABLN:07} studied a frequency-dependent SIS metapopulation model, which in our notation is given by
\begin{eqnarray*}
\frac{dx_{i}}{dt} & = & \sum_{j\neq i} R_{ji} x_{j}(t) + R_{ii} x_{i}(t) + \gamma_{i} y_{i}(t) - \frac{\beta_{i} x_{i}(t) y_{i}(t)}{x_{i}(t) + y_{i}(t)} \label{Dis:Eq2}\\ 
\frac{dy_{i}}{dt} & = & \sum_{j\neq i} Q_{ji} y_{j}(t) + Q_{ii} y_{i}(t) - \gamma_{i} y_{i}(t) + \frac{\beta_{i} x_{i}(t) y_{i}(t)}{x_{i}(t) + y_{i}(t)}. \label{Dis:Eq3}
\end{eqnarray*}
For this model, the next generation matrix is given by $ \mbox{diag}(\beta) (\mbox{diag}(\gamma) - Q^{T})^{-1} $ \citep[Lemma 2.2]{ABLN:07} so $ R_{0}$ does not depend on the migration rates of susceptible individuals. Therefore, we are unable to control the disease spread through the altering the migration rates of susceptible individuals. Although frequency dependent and density dependent contact rates are the most commonly assumed form for contact rates, it is possible to consider contact rates that are some general function of the size of the group. For these more general contact rates, we expect results similar to Theorems~\ref{thm1} and~\ref{thm2} to hold.

%% \bibliographystyle{elsarticle-harv} 


\section*{Acknowledgements}
This research is supported in part by the Australian Research Council (Centre of Excellence for Mathematical and Statistical Frontiers, CE140100049)

\end{document}